\documentclass{article}
\usepackage{graphicx}
\usepackage{bigints}
\graphicspath{Fig/}

\usepackage[utf8]{inputenc}
\usepackage[margin=1in]{geometry}
\usepackage{color}
\usepackage{enumitem}
\usepackage{authblk}
\usepackage{esint}
\usepackage{lineno}

\usepackage{hyperref}
\usepackage{hypcap}

\usepackage[ruled,vlined]{algorithm2e}
\usepackage{algorithmic}

\usepackage{float}
\usepackage{caption}
\usepackage{subcaption}
\usepackage{setspace}
\usepackage{tcolorbox}

\usepackage{amsmath,amsfonts,amssymb,mathtools, amsthm}
\usepackage[english]{babel}
\newtheorem{prop}{Proposition}
\newtheorem{theorem}{Theorem}
\newtheorem{defi}{Definition}
\newtheorem{coro}{Corollary}
\newtheorem{lemma}{Lemma}
\newtheorem{remark}{Remark}
\numberwithin{equation}{section}

\DeclareMathOperator*{\argmin}{argmin}

\usepackage{authblk}

\title{The consecutive lifting-projection flow as an approximation of Boltzmann and Landau flow}

\author[1]{Kun Huang}

\affil[1]{Department of Mathematics, Virginia Tech}
\date{}

\begin{document}
\maketitle

\begin{abstract}
    We introduce the consecutive lifting–projection (LP) flow as a novel approximation framework for the spatially homogeneous Boltzmann and Landau equations. The key idea is to lift the nonlinear collision operator to a higher dimensional linear Kac master equation on spheres, evolve this lifted equation in time, and project the solution back to the lower dimensional velocity space. The resulting LP flow is a tangent flow to the original kinetic dynamics and admits a clear semigroup structure.
    
    We show that the consecutive LP flow preserves mass, momentum, and energy, satisfies an entropy dissipation property, and converges to the correct Maxwellian equilibrium. In addition, the lifting removes the intrinsic nonlinearity of the collision operator and enables explicit analytical representations of the solution. For Maxwell molecules, we provide an $L^1$ error estimate quantifying the accuracy over finite time intervals.

    The framework provides a concise and general methodology for constructing reliable numerical solvers in kinetic theory. It unifies existing explicit discretizations, which helps understanding numerical stability and clarifying the trade-off between conservation and positivity. More importantly, it enables the development of new schemes. In particular, we propose the Green's function method, which is not only unconditionally stable, but also perfectly compatible with fast spectral discretizations.
\end{abstract}

\section{Introduction}
Boltzmann and Landau equations have been the keystone in kinetic theory. There have been two classes of numerical methods for solving these equations: the deterministic and the stochastic. Stochastic methods are easy to implement but suffers from slow convergence and noise. We will focus on the deterministic ones. They can be further categorized into spectral methods \cite{gamba2009spectral, gamba2017fast, pareschi2000fast, zhang2017conservative}, finite difference methods \cite{TAITANO2015357, shiroto2019structure}, Galerkin methods \cite{gamba2018galerkin}, particle methods \cite{carrillo2020particle,bailo2024collisional}, etc. Error analysis is hard, mostly due to the nonlinearity of the collisional operators. We have error analysis of spectral methods for Boltzmann equations \cite{filbet2011analysis,alonso2018convergence, hu2021new}, but for Landau equations there is no published work yet. 

This paper is inspired by the recent work of Guillen and Silvestre \cite{guillen2025landau}. They proposed a lifted version of Landau collision operator and were able to prove the decay of Fisher information using that form, which finally leads to regularity of the solution. Their paper \cite{guillen2025landau} is on pure analysis. However, we noticed that the lifted Landau operator not only inherits all the information of the Landau collision operator, but is also linear, so it will probably lead to better numerical methods. And in this paper we will demonstrate why that is possible and how that can be achieved.

We introduce the lifting-projection flow as an approximation of the Landau/Boltzmann flow and estimate the error. The governing equation for the lifting-projection flow is a linear master equation on spheres, therefore it is much easier to do numerical analysis and develop numerical methods for it. 

We propose a family of numerical schemes based on the lifting-projection technique.

The proposed framework brings new possibilities in two aspects: analysis of existing schemes and development of new schemes.

This lifting-projection framework also provides a new perspective which helps understanding the role of positivity in numerical solutions. We will show that when positivity contradicts conservation, it is reasonable to sacrifice positivity.

The paper will be organized as follows. In Section \ref{sec:tangent_flow} we introduce the concept of tangent flow and lifting-projection (LP) flow. Then in Section \ref{sec:semigroup} we derive the analytical representation of LP flow based on Green's function and estimate the error. A family of numerical schemes will be proposed in Section \ref{sec:scheme}. The conclusions in Section \ref{sec:tangent_flow} will be verified numerically in Section \ref{sec:verify}. Finally in Section \ref{sec:conclusion} we summarize and provide a road map for future works.

\section{Lifting-projection flow: a special tangent flow of Landau and Boltzmann equations}\label{sec:tangent_flow}

In some function space $V$, given a time-independent map $q:V\rightarrow V$ and an initial condition $f^{(0)} \in V$, the equation
\begin{equation*}
    \partial_{t}f = qf,
\end{equation*}
renders a flow $(t,f(t))$ in $V$. 

The goal of this section is to find another flow $(t,\widetilde{f}(t))$ that approximates $(t,f(t))$, but is easier to calculate.

\subsection{Tangent flows}

\begin{defi}[tangent flow]
    Let $(t, f(t))$ be the flow given by $\partial_{t} f = qf$ with initial condition $f^{(0)} \in V$. 
    If there is another function $\widetilde{f}(t)$ such that
    \begin{equation*}
        \begin{split}
            f(\cdot, t_{0}) &= \widetilde{f}(\cdot, t_{0}),\\
            \partial_{t} f(\cdot, t_{0}) &= \partial_{t}\widetilde{f}(\cdot, t_{0}),\\
        \end{split}
    \end{equation*}
    then $(t,\widetilde{f}(t))$ is a tangent flow of $q$ at $(t_{0},f(t_{0}))$.
\end{defi}

\begin{prop}[local error of tangent flow]\label{prop:error_tangent}
    Let $E$ be the error of tangent flow, i.e. $E \coloneqq \widetilde{f} - f$. For any $T > 0$ we have
    \begin{equation}\label{eq:tangent_error}
        \left\lVert E(\cdot, t_{0} + T)\right\rVert_{V} \leq \frac{1}{2} \left( \left\lVert \partial_{tt} f\right\rVert_{L^{\infty}(V,(t_{0}, t_{0} + T))} + \left\lVert \partial_{tt} \widetilde{f} \right\rVert_{L^{\infty}(V,(t_{0}, t_{0} + T))} \right) T^{2}
    \end{equation}
\end{prop}

\begin{proof}
Note that by definition,
\begin{equation*}
    E(t_{0}) = \partial_{t} E (t_{0}) = 0,
\end{equation*}
therefore the remainder of Taylor expansion can be written in integral form as follows:
    \begin{equation*}
    \begin{split}
        E(t_{0}+T) = &E(t_{0}) + \int_{t_{0}}^{t_{0}+T} \partial_{t} E (\tau) d\tau\\
    =&\int_{t_{0}}^{t_{0}+T} \partial_{t} E (\tau) d\tau\\
    =&\int_{t_{0}}^{t_{0}+T} \left[\partial_{t}E(t_{0}) + \int_{t_{0}}^{\tau} \partial_{tt} E (s)ds\right]d\tau\\
    =&\int_{t_{0}}^{t_{0}+T} \left[\int_{t_{0}}^{\tau} \partial_{tt} E (s)ds\right]d\tau.
    \end{split}
\end{equation*}
Taking its norm, we obtain
\begin{equation*}
    \begin{split}
         \left\lVert E(\cdot, t_{0} + T)\right\rVert_{V} =&  \left\lVert \int_{t_{0}}^{t_{0}+T} \left[\int_{t_{0}}^{\tau} \partial_{tt} E (s)ds\right]d\tau\right\rVert_{V}\\
         \leq& \int_{t_{0}}^{t_{0}+T} \int_{t_{0}}^{\tau} \left\lVert \partial_{tt} E\right\rVert_{L^{\infty}(V,(t_{0}, t_{0} + T))} ds d\tau \\
         \leq& \frac{1}{2}\left\lVert \partial_{tt} E\right\rVert_{L^{\infty}(V,(t_{0}, t_{0} + T))} T^{2}\\
         \leq &\frac{1}{2} \left( \left\lVert \partial_{tt} f\right\rVert_{L^{\infty}(V,(t_{0}, t_{0} + T))} + \left\lVert \partial_{tt} \widetilde{f} \right\rVert_{L^{\infty}(V,(t_{0}, t_{0} + T))} \right) T^{2}
    \end{split}
\end{equation*}
\end{proof}

\begin{defi}[consecutive tangent flow]
    Let $(t, f(t))$ be the flow given by $\partial_{t} f = qf$ with initial condition $f^{(0)}$. 

    Suppose that $(t,g_{0}(t))$ is a tangent flow of $q$ at $(0,f^{(0)})$, we denote $g_{0}(\Delta t)$ as $\widetilde{f}^{(1)}$. Now for $n\geq 1$, suppose that $(t,g_{n}(t))$ is a tangent flow of $q$ at $(n\Delta t, \widetilde{f}^{(n)})$, we denote $g_{n}(n\Delta t + \Delta t)$ as $\widetilde{f}^{(n+1)}$.

    The consecutive tangent flow $(t,\widetilde{f}(t))$ is defined as follows:
    \begin{equation}
        \widetilde{f}(t) = g_{n}(t), \ \forall\ t\in(n\Delta t, (n+1) \Delta t),
    \end{equation}
    which is the red, green and purple curve shown in Figure \ref{fig:consec_tangent}.
\end{defi}

\begin{figure}[htbp!]
    \centering
    \includegraphics[width=0.8\textwidth]{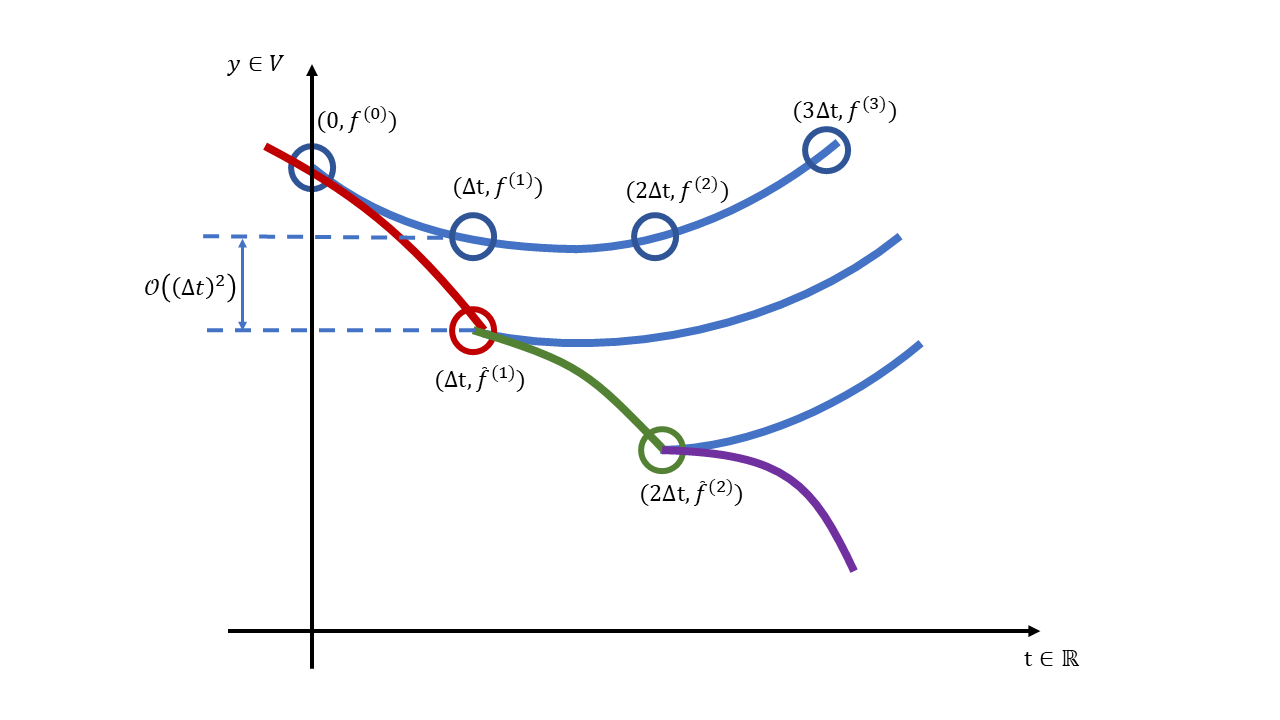}
    \caption{Consecutive tangent flow. The blue curves are flows rendered by the original equation $\partial_{t}f=qf$ with different initial conditions.}
    \label{fig:consec_tangent}
\end{figure}

As is shown in Figure \ref{fig:nonuni_tangent}, tangent flows are not unique, most of them would not make the problem easier to solve. There exists a special one that reveals the structure of Landau/Boltzmann equation: the lifting-projection (LP) flow.

\begin{figure}[htbp!]
    \centering
    \includegraphics[width=0.8\textwidth]{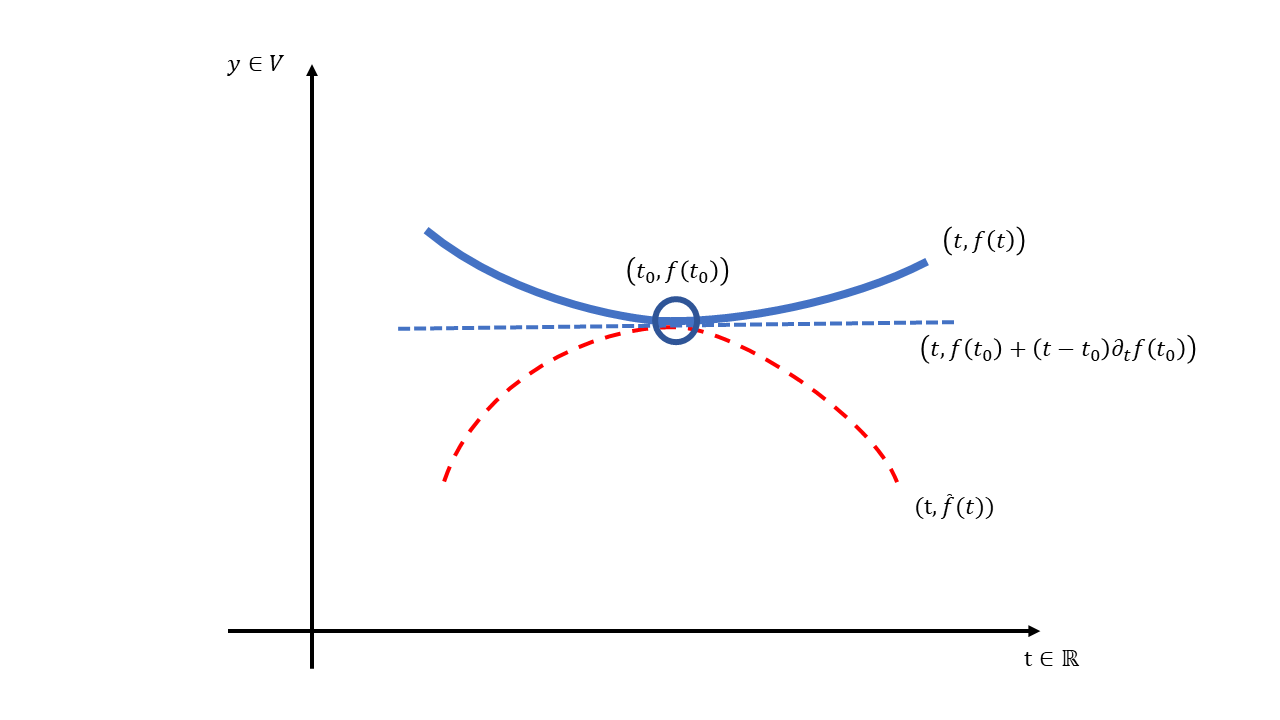}
    \caption{Non-uniqueness of tangent flows. In particular, the forward Euler flow (the blue dashed line) is a tangent flow.}
    \label{fig:nonuni_tangent}
\end{figure}

\subsection{The lifting-projection (LP) flow}

For a function $F = F(\mathbf{v},\mathbf{w})$ defined on $\mathbb{R}^d \times \mathbb{R}^d$,  we define the lifted collision operator $Q$ by
\begin{equation}\label{eq:Q_def}
    Q F \coloneqq \begin{cases}
        (\nabla_{v} - \nabla_{w})\cdot \left(S(\mathbf{v}-\mathbf{w}) \cdot (\nabla_{v} - \nabla_{w})F\right), &\text{Landau}\\[10pt]
        \int_{\mathbb{S}^{2}} B(\mathbf{v}-\mathbf{w},\mathbf{\zeta}) \left(F(\mathbf{v}',\mathbf{w}') - F(\mathbf{v},\mathbf{w})\right)d\mathbf{\zeta}, &\text{Boltzmann}
    \end{cases}
\end{equation}
where the matrix $S(\mathbf{u})$ is given by
\begin{equation*}
    S(\mathbf{u}) =|\mathbf{u}|^{\gamma} \left((\mathbf{u}^{T}\mathbf{u})\mathbb{I} - \mathbf{u} \mathbf{u}^{T}\right),
\end{equation*}
the Boltzmann collision kernel is
\begin{equation*}
    B(\mathbf{u},\zeta) = |\mathbf{u}|^{\gamma}b(\widehat{\mathbf{u}}\cdot\mathbf{\zeta}),
\end{equation*}
the pre-collision velocities are
\begin{equation*}
    \begin{cases}
        \mathbf{v}' =& \frac{1}{2}\left(\mathbf{v} + \mathbf{w} + |\mathbf{v} - \mathbf{w}| \mathbf{\zeta}\right)\\
            \mathbf{w}' =& \frac{1}{2}\left(\mathbf{v} + \mathbf{w} - |\mathbf{v} - \mathbf{w}| \mathbf{\zeta}\right).
    \end{cases}
\end{equation*}

The weak form of lifted operator $Q$ is
\begin{equation*}
    \left(QF, \Psi(\mathbf{v},\mathbf{w})\right) = \begin{cases}
        -\iint_{vw} S:((\nabla_v-\nabla_w)F\otimes(\nabla_v-\nabla_w)\Psi),&\text{Landau}\\[10pt]
       -\frac{1}{2}\iint_{vw}\int_{\mathbb{S}^{2}} B(\mathbf{v}-\mathbf{w},\mathbf{\zeta}) (F'-F)\left(\Psi' - \Psi \right)d\mathbf{\zeta},&\text{Boltzmann}
    \end{cases}
\end{equation*}

Define the projection operator $\pi$ as
\begin{equation*}
    \pi g(\mathbf{v}) =\int_{w} g(\mathbf{v}, \mathbf{w})
\end{equation*}
and denote the outer product of two functions as $f\otimes g (\mathbf{v}, \mathbf{w})\coloneqq f(\mathbf{v})g(\mathbf{w})$, then the collision operator $q(f)$ satisfies
\begin{equation}\label{eq:Q_q}
   \pi Q(f\otimes f) = q(f) \coloneqq  \begin{cases}
        \int_{w}  \nabla_{v} \cdot \left(S(\mathbf{v}-\mathbf{w}) \cdot (\nabla_{v} - \nabla_{w})f(\mathbf{v})f(\mathbf{w})\right),&\text{Landau}\\[10pt]
        \int_{w} \int_{\mathbb{S}^{2}} B(\mathbf{v}-\mathbf{w},\mathbf{\zeta}) \left(f(\mathbf{v}')f(\mathbf{w}') - f(\mathbf{v})f(\mathbf{w})\right)d\mathbf{\zeta},&\text{Boltzmann}
    \end{cases}
\end{equation}

\bigskip

\begin{lemma}[lifting-projection flow]
On the time interval $(t_{0}, t_{0}+T)$, solve the lifted equation
\begin{equation}\label{eq:2-kac}
    \partial_{t} F = Q F
\end{equation}
with initial condition $\left.F(\mathbf{v},\mathbf{w},t)\right|_{t=t_{0}} =\left.f(v,t)\right|_{t=t_{0}}\left.f(w,t)\right|_{t=t_{0}}$.

Then the projection of the lifted solution $(t, \pi F (\cdot, t))$ is a tangent flow of collision operator $q$ at $(t_{0}, f(\cdot, t_{0}))$. We call it the lifting-projection(LP) flow of $q$.
\end{lemma}
\begin{proof}
    Use Equation (\ref{eq:Q_q}).
\end{proof}

\begin{remark}
    Equation (\ref{eq:2-kac}) is called the 2-particle Kac master equation in literature.
\end{remark}

\begin{defi}[Consecutive LP flow]
    On the time interval $[n\Delta t, (n+1)\Delta t)$, let $F$ be the solution to the lifted equation 
    \begin{equation*}
    \partial_{t} F = Q F,
    \end{equation*}
    with initial condition $\left.F(\mathbf{v},\mathbf{w},t)\right|_{t=n\Delta t} =\widetilde{f}^{(n)}(v)\widetilde{f}^{(n)}(w)$. 
    
    Denote $\pi F((n+1)\Delta t)$ as $\widetilde{f}^{(n+1)}$, we will use it to construct the initial condition for the next time interval.

    The consecutive 2-particle Kac flow is defined as
    \begin{equation*}
        F_{K}(t) = F(t), \ \forall\ t\in[n\Delta t, (n+1) \Delta t).
    \end{equation*}
    
    The consecutive LP flow $(t,\widetilde{f}(t))$ is defined as follows:
    \begin{equation}
        \widetilde{f}(t) = \pi F(t), \ \forall\ t\in[n\Delta t, (n+1) \Delta t).
    \end{equation}
\end{defi}

\begin{remark}
    The consecutive LP flow is continuous in time. However, the consecutive 2-particle Kac flow is not continuous in time, in fact
    \begin{equation}
        F_{K}(n\Delta t) = \left( \lim_{t\rightarrow (n\Delta t)^{-}}\pi F_{K}(t)\right) \otimes \left( \lim_{t\rightarrow (n\Delta t)^{-}}\pi F_{K}(t)\right) \neq \lim_{t\rightarrow (n\Delta t)^{-}} F_{K}(t).
    \end{equation}
\end{remark}

\subsection{A simple example of the LP flow}
To help illustrating the concept of LP flow, we will present a simple example in what follows.

Let operator $q_{s}$ be the loss part of a Boltzmann operator, and denote its lifted version as $Q_{s}$:
\begin{equation}\label{eq:q_loss}
    \begin{split}
        q_{s} f &\coloneqq - \int_{w} f(v) f(w),\\
        Q_{s} F &\coloneqq - F(v ,w).
    \end{split}
\end{equation}

Suppose that we have unit initial mass $\int_{v} f^{(0)}(v) = 1$. The solution to $\partial_{t}f = q_{s} f$ with initial condition $\left.f(v,t)\right|_{t=0} = f^{(0)}(v)$ is
\begin{equation*}
    f(v,t) = \frac{1}{1+t}f^{(0)}(v).
\end{equation*}
And the solution to $\partial_{t} F = Q_{s} F$ with initial condition $\left.F(v,w,t)\right|_{t=0} = f^{(0)}(v)f^{(0)}(w)$ is
\begin{equation*}
    F(v,w,t) = e^{-t}f^{(0)}(v)f^{(0)}(w).
\end{equation*}
By definition, $(t, \pi F)$ is a tangent flow of $q_{s}$ at $(0,f^{(0)}(v))$, note that
\begin{equation*}
    \pi F (v,t)= e^{-t} f^{(0)}(v) \neq f(v,t).
\end{equation*}

\begin{remark}
    This is just a toy model since neither $q_{s}$ or $Q_{s}$ preserve the mass of solution, indeed
    \begin{align*}
        \int_{v} f(v,t) &= \frac{1}{1+t} <1, &\ \forall \  t>0;\\
            \iint_{vw} F(v,w,t) &= e^{-t}<1,&\ \forall \  t>0.
    \end{align*}
    
\end{remark}

\subsection{Properties of the lifting-projection flow} \label{subsec:properties}

The key advantage of the lifting-projection flow is that its governing equation is a linear equation. But that property alone does not guarantee a good numerical method. In this subsection, we will find out more.

Define $\mathbf{z} = \mathbf{v}+\mathbf{w}$ and $\mathbf{u} = \mathbf{v}-\mathbf{w}$, any function $G(\mathbf{v},\mathbf{w})$ can also be expressed as
\begin{equation*}
    G(\mathbf{z},\mathbf{u}) = G(\frac{1}{2}\mathbf{z}+\frac{1}{2}\mathbf{u},\frac{1}{2}\mathbf{z}-\frac{1}{2}\mathbf{u}).
\end{equation*}
If $G(\mathbf{v},\mathbf{w}) \leq e^{-v^{2}-w^{2}}$ for any $(\mathbf{v},\mathbf{w}) \in \mathbb{R}^{2d}$ then $G(\mathbf{z},\mathbf{u}) \leq e^{-z^{2}/2-u^{2}/2}$ for any $(\mathbf{z},\mathbf{u}) \in \mathbb{R}^{2d}$.

Define $\rho = |\mathbf{u}|$ and $\sigma = \mathbf{u}/|\mathbf{u}|$, for Landau collision operators we have \cite{guillen2025landau}:
\begin{equation}\label{eq:Q_sphere_diff}
    Q F = 4\rho^{\gamma}\Delta_{\sigma} F(\rho, \sigma;\mathbf{z},t)
\end{equation}
where $\Delta_{\sigma}$ is the Laplace-Beltrami operator with respect to $\sigma$ on $\mathbb{S}^{d-1}$. 

Indeed, the lifted Landau operator is
\begin{equation*}
    Q F(\mathbf{u}, \mathbf{z}) = 4|\mathbf{u}|^{\gamma}\nabla_{u}\cdot\left(\left((\mathbf{u}^{T}\mathbf{u})\mathbb{I}-\mathbf{u}\mathbf{u}^{T}\right)\cdot\nabla_{u}F\right)
\end{equation*}

When $d=2$, using polar coordinates $(\rho, \phi)$ for $\mathbf{u}$, 
\begin{equation*}
    \begin{split}
        \frac{1}{4}QF(\rho,\phi, \mathbf{z}) =& |\mathbf{u}|^{\gamma}\nabla_{u}\cdot\left(\left((\mathbf{u}^{T}\mathbf{u})\mathbb{I}-\mathbf{u}\mathbf{u}^{T}\right)\cdot\nabla_{u}F\right)\\
        =& \rho^{\gamma}\nabla_{u}\cdot\left(\rho^{2}\left(\mathbb{I}-\hat{\rho}\hat{\rho}^{T}\right)\cdot\left((\partial_{\rho}F)\hat{\rho} + \frac{1}{\rho}(\partial_{\phi}F)\hat{\phi}\right)\right)\\
        =&\rho^{\gamma}\nabla_{u}\cdot\left(\rho^{2}\left(\frac{1}{\rho}(\partial_{\phi}F)\hat{\phi}\right)\right)\\
        =&\rho^{\gamma}\nabla_{u}\cdot\left(\rho(\partial_{\phi}F)\hat{\phi}\right)\\
        =&\rho^{\gamma}\frac{1}{\rho}\partial_{\phi}(\rho\partial_{\phi}F)\\
        =&\rho^{\gamma}\partial_{\phi}^{2}F.
    \end{split}
\end{equation*}
Recall that
\begin{equation*}
    \Delta_{u} F =  \frac{1}{\rho}\partial_{\rho}(\rho \partial_{\rho}F) + \frac{1}{\rho^{2}}\partial_{\phi\phi}F,
\end{equation*}
hence we have
\begin{equation*}
    \partial_{\phi}^{2}F = \rho^{2}\Delta_{u} F - \rho\partial_{\rho}(\rho\partial_{\rho}F).
\end{equation*}
Now we can see that the operator $Q$ is a homogeneous diffusion operator on $\theta$ for any fixed $(\rho,\mathbf{z})$, with diffusion coefficient being $\rho^{\gamma}$.

\bigskip
When $d = 3$, using spherical coordinates $(\rho, \theta, \phi)$ for $\mathbf{u}$,
\begin{equation*}
    \begin{split}
        \frac{1}{4}QF(\rho, \theta, \phi,\mathbf{z}) = &|\mathbf{u}|^{\gamma}\nabla_{u}\cdot\left(\left((\mathbf{u}^{T}\mathbf{u})\mathbb{I}-\mathbf{u}\mathbf{u}^{T}\right)\cdot\nabla_{u}F\right)\\
        =&\rho^{\gamma}\nabla_{u}\cdot\left(\rho^{2}\left(\mathbb{I}-\hat{\rho}\hat{\rho}^{T}\right)\cdot\left((\partial_{\rho}F)\hat{\rho} + \frac{1}{\rho}(\partial_{\theta}F)\hat{\theta} + \frac{1}{\rho \sin{\theta}}(\partial_{\phi}F)\hat{\phi}\right)\right)\\
        =&\rho^{\gamma}\nabla_{u}\cdot\left(\rho^{2}\left( \frac{1}{\rho}(\partial_{\theta}F)\hat{\theta} + \frac{1}{\rho \sin{\theta}}(\partial_{\phi}F)\hat{\phi}\right)\right)\\
        =&\rho^{\gamma}\nabla_{u}\cdot\left(\rho\left((\partial_{\theta}F)\hat{\theta} + \frac{1}{\sin{\theta}}(\partial_{\phi}F)\hat{\phi}\right)\right)\\
        =&\rho^{\gamma}\left( \frac{1}{\rho \sin{\theta}}\partial_{\theta}(\rho(\partial_{\theta}F)\sin{\theta}) + \frac{1}{\rho \sin{\theta}}\partial_{\phi} (\rho\frac{1}{\sin{\theta}}(\partial_{\phi}F))\right)\\
        =&\rho^{\gamma}\left( \frac{1}{ \sin{\theta}}\partial_{\theta}((\partial_{\theta}F)\sin{\theta}) + \frac{1}{\sin^{2}{\theta}}\partial_{\phi}^{2} F\right)\\
        =&\rho^{\gamma}\Delta_{\sigma}F.
    \end{split}
\end{equation*}

For Boltzmann collision operator we have
\begin{equation*}
    Q F=\rho^{\gamma} J_{\sigma}F(\sigma, \rho;\mathbf{z},t) .
\end{equation*}
where
\begin{equation*}
    J_{\sigma} F \coloneqq \int_{\mathbb{S}^{d-1}} b(\zeta\cdot\sigma) F(\zeta, \rho;\mathbf{z},t) d\zeta- \left(\int_{\mathbb{S}^{d-1}} b(\zeta\cdot\sigma) d\zeta\right) F(\sigma, \rho;\mathbf{z},t)
\end{equation*}

For detailed derivation, see Appendix.

\bigskip
The $Q$-flow does not lead to Gaussian equilibrium, therefore we know that $\pi F \neq f$ when $t>0$. Nevertheless, it turns out that $\pi F$ has the same conservation laws as $f$.
\begin{prop}[LP flow preserves conservation laws]\label{prop:conservation}
    Take test functions $\psi(\mathbf{v}) \in \left\{1,\mathbf{v},|\mathbf{v}|^{2}\right\}$,  we have
    \begin{equation*}
        (\pi F( T), \psi) = (f^{(0)}, \psi)
    \end{equation*}
    for any $T \in \mathbb{R}^{+}$.
\end{prop}
\begin{proof}
    For $\psi(\mathbf{v}) \in \left\{1,\mathbf{v},|\mathbf{v}|^{2}\right\}$, we have
    \begin{equation*}
        \partial_{t}\left(F, \psi(\mathbf{v})+\psi(\mathbf{w})\right) = \left(QF, \psi(\mathbf{v})+\psi(\mathbf{w})\right) = 0,
    \end{equation*}
    hence
    \begin{equation*}
        \left(F(T), \psi(\mathbf{v})+\psi(\mathbf{w})\right) = \left(f^{(0)}\otimes f^{(0)} , \psi(\mathbf{v})+\psi(\mathbf{w})\right).
    \end{equation*}
    Since $QF(\mathbf{v},\mathbf{w}) = QF(\mathbf{w},\mathbf{v})$, we know that
    \begin{equation*}
        F(\mathbf{v},\mathbf{w}, T) = F(\mathbf{w},\mathbf{v}, T),\ \forall T\in \mathbb{R}^{+}.
    \end{equation*}
    Consequently
    \begin{equation*}
        \begin{split}
            \left(\pi F(T), \psi\right) \coloneqq& \int_{v}\left(\int_{w} F(\mathbf{v},\mathbf{w},T)\right)\psi(\mathbf{v}) \\
        = &\iint_{vw}F(\mathbf{v},\mathbf{w},T)\psi(\mathbf{v}) \\
        =& \frac{1}{2}\left(\iint_{vw}F(\mathbf{v},\mathbf{w},T)\psi(\mathbf{v}) + \iint_{vw}F(\mathbf{w},\mathbf{v},T)\psi(\mathbf{w})\right)\\
        =&\frac{1}{2}\left(F(T), \psi(\mathbf{v})+\psi(\mathbf{w})\right) \\
        =&\frac{1}{2}\left(f^{(0)}\otimes f^{(0)} , \psi(\mathbf{v})+\psi(\mathbf{w})\right)\\
        =&\left( f^{(0)}, \psi\right) 
        \end{split}
    \end{equation*}
\end{proof}

\begin{theorem}[H-theorem for consecutive LP flow]\label{thm:entropy}
Define the entropy $H$ of distribution functions as:
\begin{equation*}
    \begin{split}
        H(f(\mathbf{v})) \coloneqq & \int_{v} f(\mathbf{v})\log{f(\mathbf{v})},\\
        H(F(\mathbf{v},\mathbf{w})) \coloneqq & \iint_{vw} F(\mathbf{v},\mathbf{w})\log{F(\mathbf{v},\mathbf{w})}.
    \end{split}
\end{equation*}
    Let $(t, F_{K}(t))$ be the consecutive 2-particle Kac flow, and $(t, \widetilde{f}(t))$ be the consecutive LP flow. 
    
    For any $0\leq t_{1}<t_{2}<+\infty$ we have
    \begin{equation*}
            H(F_{K}(\cdot, t_{2})) \leq H(F_{K}(\cdot, t_{1})).
    \end{equation*}
    In addition, for any $0\leq n < m < \infty$ we have
    \begin{equation*}
        H(\widetilde{f}^{(m)}) \leq H(\widetilde{f}^{(n)}). 
    \end{equation*}
\end{theorem}
\begin{proof}
    Within each interval we know that the 2-particle Kac master equation satisfy the H-theorem:
    \begin{equation*}
        \partial_{t} H(F) \leq 0.
    \end{equation*}
    It suffices to show that
    \begin{equation*}
        H\left(F_{K}(n\Delta t)\right)  \leq H\left(\lim_{t\rightarrow (n\Delta t)^{-}} F_{K}(t)\right).
    \end{equation*}

    Note that
    \begin{equation*}
        F_{K}(n\Delta t) = \left( \lim_{t\rightarrow (n\Delta t)^{-}}\pi F_{K}(t)\right) \otimes \left( \lim_{t\rightarrow (n\Delta t)^{-}}\pi F_{K}(t)\right) \neq \lim_{t\rightarrow (n\Delta t)^{-}} F_{K}(t),
    \end{equation*}
    hence it remains to show that for any $F$ s.t. $F(\mathbf{v}, \mathbf{w}) = F(\mathbf{w}, \mathbf{v})$ we have
    \begin{equation*}
        H(\pi F \otimes \pi F) \leq H(F).
    \end{equation*}

    Recall that the mutual information is always non-negative:
    \begin{equation*}
        I(F) \coloneqq \iint_{vw} F \log{\frac{F}{\pi F\otimes \pi F}} \geq 0.
    \end{equation*}
    Decompose the mutual information as follows
    \begin{equation*}
    \begin{split}
\iint_{vw} F \log{\frac{F}{\pi F\otimes \pi F}} = &\iint_{vw}F\log{F} - \iint_{vw} F \log{\pi F(\mathbf{v})} - \iint_{vw} F \log{\pi F(\mathbf{w})}\\
=& \iint_{vw}F\log{F} - \int_{v} \left(\int_{w}F\right) \log{\pi F(\mathbf{v})} - \int_{w} \left(\int_{v}F\right) \log{\pi F(\mathbf{w})}\\
=& \iint_{vw}F\log{F} - 2\int_{v} (\pi F) \log{(\pi F)}
    \end{split}
    \end{equation*}

    Now we can see that 
    \begin{equation*}
        H(F) \geq 2 H(\pi F) = H(\pi F \otimes \pi F).
    \end{equation*}
    Consequently
    \begin{equation*}
        H\left(F_{K}(n\Delta t)\right) = H\left(\left( \lim_{t\rightarrow (n\Delta t)^{-}}\pi F_{K}(t)\right) \otimes \left( \lim_{t\rightarrow (n\Delta t)^{-}}\pi F_{K}(t)\right)\right) \leq H\left(\lim_{t\rightarrow (n\Delta t)^{-}} F_{K}(t)\right).
    \end{equation*}

    Recall that $\widetilde{f}^{(n)} = \lim_{t\rightarrow (n\Delta t)^{-}}\pi F_{K}(t)$ and $2 H(\pi F) = H(\pi F \otimes \pi F)$, therefore the previous conclusion
    \begin{equation*}
        H\left(\left( \lim_{t\rightarrow (m\Delta t)^{-}}\pi F_{K}(t)\right) \otimes \left( \lim_{t\rightarrow (m\Delta t)^{-}}\pi F_{K}(t)\right)\right) \leq  H\left(\left( \lim_{t\rightarrow (n\Delta t)^{-}}\pi F_{K}(t)\right) \otimes \left( \lim_{t\rightarrow (n\Delta t)^{-}}\pi F_{K}(t)\right)\right)
    \end{equation*}
    implies that
    \begin{equation*}
        H(\widetilde{f}^{(m)}) \leq H(\widetilde{f}^{(n)}).
    \end{equation*}
    
    \end{proof}

\begin{coro}
    The consecutive LP flow will finally converge to the same Gaussian distribution equilibrium as the original q-flow.
\end{coro}


\section{The semigroup generated by the lifted collision operator}\label{sec:semigroup}
On each sphere the master equation is linear and homogeneous, therefore an explicit analytical representation of the solution can be readily derived. As we will see later, such representation not only enables error analysis but also provides useful insight for the numerical implementation.

\subsection{The lifted Landau equation}
\subsubsection*{2D case}
Recall that the lifted 2d Landau equation reads
\begin{equation*}
    \partial_{t} F = QF = 4\rho^{\gamma}\Delta_{\phi} F(\rho, \phi;\mathbf{z},t).
\end{equation*}
Given initial condition $F_{0}(\rho, \phi; \mathbf{z})$, the solution to the above equation is
\begin{equation*}
    F(\rho,\phi,t;\mathbf{z}) = \exp(Qt) F_{0}= \int_{0}^{2\pi} K(\rho,\phi-\xi,t)F_{0}(\rho, \xi;\mathbf{z})d\xi,
\end{equation*}
where the Green's function for diffusion on a circle reads
\begin{equation*}
    K(\rho,\phi-\xi,t) = \frac{1}{2\pi} \sum_{k=-\infty}^{\infty} \exp(-k^{2}4\rho^{\gamma}t)\exp(ik(\phi-\xi)).
\end{equation*}
Consequently, the LP flow associated to 2d Landau equation yields
\begin{equation}\label{eq:2dLandau_LP}
    \widetilde{f}^{(n+1)} (\mathbf{v}) = \int_{w}\int_{\mathbb{S}^{1}} K(\rho,\cos^{-1}(\widehat{\mathbf{u}}\cdot \mathbf{\zeta}),\Delta t)\widetilde{f}^{(n)}(\mathbf{v}') \widetilde{f}^{(n)}(\mathbf{w}') d\zeta
\end{equation}

\subsubsection*{3D case}
Recall that the lifted 3d Landau equation reads
\begin{equation*}
    \partial_{t} F = QF = 4\rho^{\gamma}\Delta_{\sigma} F(\rho, \sigma;\mathbf{z},t).
\end{equation*}
Given initial condition $F_{0}(\rho, \sigma; \mathbf{z})$, the solution to the above equation is
\begin{equation*}
    F(\rho, \sigma, t; \mathbf{z}) = \exp(Qt) F_{0}= \int_{\mathbb{S}^{2}}K(\rho,\zeta\cdot\sigma, t)F_{0}(\rho, \zeta;\mathbf{z}) d\zeta,
\end{equation*}
where the Green's function for diffusion on a sphere reads
\begin{equation*}
    K(\rho,\zeta\cdot\sigma, t) = \frac{1}{4\pi}\sum_{l=0}^{\infty} \exp(-l(l+1)4\rho^{\gamma}t)P_{l}(\zeta\cdot\sigma),
\end{equation*}
and $P_{l}$ is the Legendre polynomial.

Consequently, the LP flow associated to 3d Landau equation yields
\begin{equation}\label{eq:3dLandau_LP}
    \widetilde{f}^{(n+1)} (\mathbf{v}) = \int_{w}\int_{\mathbb{S}^{2}} K(\rho,\widehat{\mathbf{u}}\cdot \mathbf{\zeta},\Delta t)\widetilde{f}^{(n)}(\mathbf{v}') \widetilde{f}^{(n)}(\mathbf{w}') d\zeta.
\end{equation}



\bigskip
\begin{remark}
    From (\ref{eq:2dLandau_LP}) and (\ref{eq:3dLandau_LP}) we can observe that, both relations are in the same form as the gain part of a Boltzmann operator. As a consequence, any discretization of Boltzmann operator can be translated straightforwardly.
\end{remark}

\subsection{The lifted VHS Boltzmann equation}
In what follows we focus on a specific family of Boltzmann collision operators: the variable hard sphere (VHS) model, whose kernel is independent on the angular variable:
\begin{equation*}
    B(\mathbf{v}-\mathbf{w},\sigma) \equiv \frac{1}{\int_{\mathbb{S}^{d-1}}1d\zeta}\left\vert\mathbf{v}-\mathbf{w}\right\vert^{\gamma}.
\end{equation*}
Define $\mathbf{u} = \mathbf{v} - \mathbf{w}$, $\mathbf{z} = \mathbf{v} + \mathbf{w}$, $\rho = |\mathbf{u}|$, and $\mathbf{\sigma} = \mathbf{u}/|\mathbf{u}|$. 

Let $F = F(\rho, \sigma; \mathbf{z},t)$, and define spherical averaging operator as follows:
\begin{equation*}
    A F(\rho, \sigma;\mathbf{z}) \coloneqq \frac{\int_{\mathbb{S}^{d-1}}  F(\rho, \zeta;\mathbf{z}) d\zeta}{\int_{\mathbb{S}^{d-1}}1d\zeta},
\end{equation*}
then the lifted Boltzmann equation reads
\begin{equation}\label{eq:simpleBoltz}
    \partial_{t} F = QF =\rho^{\gamma} \left(A F - F\right).
\end{equation}

For each fixed $\rho$ and $\mathbf{z}$, define a rescaled time variable $s = \rho^{\gamma}t$, we obtain
\begin{equation*}
    \partial_{s}F = AF - F
\end{equation*}
Applying the method of integrating factors,
\begin{equation*}
    \partial_{s}\left(e^{s}F\right) = A e^{s}F,
\end{equation*}
it follows that
\begin{equation*}
    e^{s} F(s) = \exp(As) e^{0} F_{0} = \exp(As)F_{0},
\end{equation*}
and the solution $F(s)$ can be written as
\begin{equation*}
    F(s) = e^{-s}\exp(As)F_{0}.
\end{equation*}

Now it remains to derive the operator $\exp(As)$. Note that the spherical averaging operator $A$ is a projection operator, meaning $A^{2} = A$. By definition we have
\begin{equation*}
    \exp(As) \coloneqq  I + \sum_{n=1}^{\infty} \frac{(As)^{n}}{n!} = I + A \sum_{n=1}^{\infty} \frac{s^{n}}{n!}=I + A(e^{s}-1),
\end{equation*}
hence
\begin{equation*}
    F(s) = e^{-s}\exp(As)F_{0} = e^{-s}\left(I + A(e^{s}-1)\right)F_{0} = e^{-s}F_{0} + (1-e^{-s})AF_{0}.
\end{equation*}

Recall that $s= \rho^{\gamma}t$, therefore given initial condition $F_{0}(\rho, \sigma; \mathbf{z})$, the solution to the lifted Boltzmann equation is 
\begin{equation}\label{eq:sol_to_Boltz}
    F(t) =\exp{(Qt)}F_{0}= e^{-\rho^{\gamma}t} F_{0}+ \left(1-e^{-\rho^{\gamma}t}\right)AF_{0}
\end{equation}

Consequently, the LP flow associated to VHS Boltzmann equation yields
\begin{equation}\label{eq:VHSBoltzmann_LP}
    \widetilde{f}^{(n+1)}(\mathbf{v}) = \widetilde{f}^{(n)}(\mathbf{v}) + \int_{w}\int_{\mathbb{S}^{d-1}}K(\rho, \widehat{\mathbf{u}}\cdot\zeta, \Delta t) \left(\widetilde{f}^{(n)}(\mathbf{v}')\widetilde{f}^{(n)}(\mathbf{w}') - \widetilde{f}^{(n)}(\mathbf{v})\widetilde{f}^{(n)}(\mathbf{w})\right)d\zeta,
\end{equation}
where
\begin{equation*}
    K(\rho, \widehat{\mathbf{u}}\cdot\zeta, \Delta t) = \left(1-e^{-\rho^{\gamma}\Delta t}\right)\frac{1}{\int_{\mathbb{S}^{d-1}}1d\zeta}.
\end{equation*}

\bigskip
In particular, for Maxwell molecules, i.e. when $\gamma = 0$, we have
\begin{equation*}
    \widetilde{f}^{(n+1)}(\mathbf{v}) = \widetilde{f}^{(n)}(\mathbf{v}) + (1-e^{-\Delta t})\int_{w}\int_{\mathbb{S}^{d-1}}B(\mathbf{v}-\mathbf{w},\sigma)\left(\widetilde{f}^{(n)}(\mathbf{v}')\widetilde{f}^{(n)}(\mathbf{w}') - \widetilde{f}^{(n)}(\mathbf{v})\widetilde{f}^{(n)}(\mathbf{w})\right)d\zeta.
\end{equation*}

It turns out that in this case the LP flow is just replacing the time step $\Delta t$ with $(1-e^{-\Delta t})$ in a forward Euler scheme.

\subsection{Analysis on Boltzmann equation for Maxwell molecules}

In this subsection we present error analysis on the consecutive LP flow associated to Boltzmann equation for Maxwell molecules ($\gamma = 0$).

\begin{lemma}
    If $\gamma = 0$, then there exists a constant $C_{Q}$ such that
    \begin{equation*}
        \frac{1}{t}\left\lVert\left(\exp{(Qt)}- I\right)F\right\rVert_{L^{1}} \leq C_{Q}\left\lVert F\right\rVert_{L^{1}}
    \end{equation*}
\end{lemma}
\begin{proof}
    By definition we have
    \begin{equation*}
        \begin{split}
            \frac{1}{t}\left\lVert\left(\exp{(Qt)}- I\right)F\right\rVert_{L^{1}}  =&\left\lVert \frac{1-e^{-t}}{t}(A-I)F\right\rVert_{L^{1}} \\
            \leq & \left\lVert (A-I)F\right\rVert_{L^{1}}\\
            \leq & \left\lVert F\right\rVert_{L^{1}}
        \end{split}
    \end{equation*}
\end{proof}
\begin{theorem}[Error analysis of consecutive LP flow]
    The error of consecutive LP flow is bounded as follows:
    \begin{equation*}
        \left\lVert \widetilde{f}(T)- f(T)\right\rVert_{L^{1}} \lesssim 
        e^{\lambda T}\Delta t,
    \end{equation*}
    where $\lambda$ is a constant that depends on the collision kernel.
\end{theorem}
\begin{proof}
    Recall that the consecutive Kac flow forms a semigroup within each interval $[n\Delta t,(n+1)\Delta t)$, i.e.
    \begin{equation*}
        \widetilde{f}^{(n+1)} = \pi \exp{(\Delta t Q)} \left(\widetilde{f}^{(n)}\otimes \widetilde{f}^{(n)}\right).
    \end{equation*}
    Define an auxiliary function $F_{*}$ on $[n\Delta t, (n+1)\Delta t)$ as follows:
    \begin{equation*}
        F_{*}(t)= \exp{((t-n\Delta t) Q)} \left(f^{(n)}\otimes f^{(n)}\right),
    \end{equation*}
    and let $F_{*}^{(n+1)}$ be $\exp{(\Delta t Q)} \left(f^{(n)}\otimes f^{(n)}\right)$.

    Decompose the error into two parts:
    \begin{equation*}
        \begin{split}
            \left\lVert \widetilde{f}^{(n+1)}- \pi F_{*}^{(n+1)} +  \pi F_{*}^{(n+1)} - f^{(n+1)}\right\rVert_{L^{1}} = &\left\lVert \widetilde{f}^{(n+1)}- \pi F_{*}^{(n+1)} +  \pi F_{*}^{(n+1)} - f^{(n+1)}\right\rVert_{L^{1}}\\
            \leq &\left\lVert \widetilde{f}^{(n+1)}- \pi F_{*}^{(n+1)}\right\rVert_{L^{1}} + \left\lVert \pi F_{*}^{(n+1)} - f^{(n+1)}\right\rVert_{L^{1}}\\
            =& E_{stability} + E_{local}.
        \end{split}
    \end{equation*}

    Let us estimate the $E_{stability}$ term first. Since the projection operator $\pi$ is non-expansive in $L^{1}$, we have:
    \begin{equation*}
        \begin{split}
            E_{stability} \coloneqq &\left\lVert \widetilde{f}^{(n+1)}- \pi F_{*}^{(n+1)}\right\rVert_{L^{1}}\\
            =& \left\lVert \widetilde{f}^{(n+1)} - \widetilde{f}^{(n)}- \left(\pi F_{*}^{(n+1)}-f^{(n)}\right) + \left(\widetilde{f}^{(n)} - f^{(n)}\right)\right\rVert_{L^{1}}\\
            \leq& \left\lVert \widetilde{f}^{(n+1)} - \widetilde{f}^{(n)}- \left(\pi F_{*}^{(n+1)}-f^{(n)}\right)\right\rVert_{L^{1}} + \left\lVert \left(\widetilde{f}^{(n)} - f^{(n)}\right)\right\rVert_{L^{1}}\\
            \leq& \left\lVert \left(\exp{(\Delta t Q)} - I \right)\left(\widetilde{f}^{(n)}\otimes \widetilde{f}^{(n)} - f^{(n)}\otimes f^{(n)}\right) \right\rVert_{L^{1}} + \left\lVert \widetilde{f}^{(n)} - f^{(n)}\right\rVert_{L^{1}}\\
            \leq& C_{Q}\Delta t \left\lVert\widetilde{f}^{(n)}\otimes \widetilde{f}^{(n)} - f^{(n)}\otimes f^{(n)} \right\rVert_{L^{1}} + \left\lVert \widetilde{f}^{(n)} - f^{(n)}\right\rVert_{L^{1}}\\
            \leq& \left(C_{Q}\Delta t \left(\left\lVert \widetilde{f}^{(n)}\right\rVert_{L^{1}}+ \left\lVert  f^{(n)}\right\rVert_{L^{1}}\right) + 1\right)\left\lVert \widetilde{f}^{(n)} - f^{(n)}\right\rVert_{L^{1}}\\
            =&\left(2 C_{Q}\Delta t  + 1\right)\left\lVert \widetilde{f}^{(n)} - f^{(n)}\right\rVert_{L^{1}}.
        \end{split}
    \end{equation*}

    As for the bound of $E_{local}$, recall Proposition \ref{prop:error_tangent}, it remains to show that $\left\lVert \partial_{tt}\pi F_{*}\right\rVert_{L^{\infty}(L^{1},(n\Delta t, (n+1)\Delta t)))} $ and $\left\lVert \partial_{tt} f \right\rVert_{L^{\infty}(L^{1},(n\Delta t, (n+1)\Delta t)))} $ are bounded. The time regularity of $f$ as a solution to the Boltzmann equation can be found in the literature, for example \cite{desvillettes2000spatially}. Recall (\ref{eq:sol_to_Boltz}) we have
    \begin{equation*}
        \begin{split}
            &\left\lVert \partial_{tt}\pi F_{*}\right\rVert_{L^{\infty}(L^{1},(n\Delta t, (n+1)\Delta t)))} + \left\lVert \partial_{tt}f\right\rVert_{L^{\infty}(L^{1},(n\Delta t, (n+1)\Delta t)))} \\
            \leq& \left\lVert \partial_{tt} F_{*}\right\rVert_{L^{\infty}(L^{1},(n\Delta t, (n+1)\Delta t)))} + \left\lVert \partial_{tt}f\right\rVert_{L^{\infty}(L^{1},(n\Delta t, (n+1)\Delta t)))}\\
            =&\left\lVert QQ F_{*}\right\rVert_{L^{\infty}(L^{1},(n\Delta t, (n+1)\Delta t)))} + \left\lVert \partial_{tt}f\right\rVert_{L^{\infty}(L^{1},(n\Delta t, (n+1)\Delta t)))}\\
            =& \left\lVert -e^{-(t-n\Delta t)} Q\left(f^{(n)}\otimes f^{(n)}\right)\right\rVert_{L^{\infty}(L^{1},(n\Delta t, (n+1)\Delta t)))} + \left\lVert \partial_{tt}f\right\rVert_{L^{\infty}(L^{1},(n\Delta t, (n+1)\Delta t)))}\\
            \leq & C_{q}(f^{(0)}).
        \end{split}
    \end{equation*}

    To conclude,
    \begin{equation*}
        \begin{split}
            &\left\lVert \widetilde{f}^{(n+1)} - f^{(n+1)}\right\rVert_{L^{1}} \leq \left(2 C_{Q}\Delta t  + 1\right)\left\lVert \widetilde{f}^{(n)} - f^{(n)}\right\rVert_{L^{1}} + \frac{C_{q}}{2} (\Delta t)^{2}\\
            \Rightarrow & \left\lVert \widetilde{f}^{(n)} - f^{(n)}\right\rVert_{L^{1}} \lesssim e^{2C_{Q}n\Delta t}\left(\frac{C_{q}}{4C_{Q}}\Delta t \right).
        \end{split}
    \end{equation*}
\end{proof}

\section{Numerical schemes based on lifting and projection}\label{sec:scheme}

In this section we present a novel family of numerical methods based on lifting and projection.

\subsection{Formulation}

Given a discrete solution $f_{h}^{(n)}$, we propose the following scheme to obtain $f^{(n+1)}_{h}$.

\begin{enumerate}
    \item Solve the linear Kac master equation (lifted Landau/Boltzmann)
    \begin{equation}\label{eq:discrete_lifted}
    F^{(n+1)}_{h} = f^{(n)}_{h}\otimes f^{(n)}_{h} + \sqint_{n\Delta t}^{(n+1)\Delta t } Q_{h}F_{h}(\tau) d\tau,
    \end{equation}
    where $F_{h}(\tau)$ solves the equation $\partial_{t} F_{h} = Q_{h} F_{h}$ on $(n \Delta t, (n+1)\Delta t)$ with initial condition $f^{(n)}_{h}\otimes f^{(n)}_{h}$, and $\sqint$ represents the numerical integration in time.
    \item Apply projection
    \begin{equation*}
        f^{(n+1)}_{h} =\pi F^{(n+1)}_{h}.
    \end{equation*}
\end{enumerate}

\begin{remark}[Discussion on positivity]
    Positivity of the distribution function $f$ and conservation of mass, momentum and energy are two key properties that people try to preserve when solving Landau/Boltzmann equations numerically.

The reason we need positivity is twofold. Firstly, the entropy is only well-defined for positive functions. Secondly, the diffusion term in Landau equation depends on $f$ itself. And when $f$ is not non-negative, that diffusion coefficient can lose positive semi-definiteness, which may cause numerical instability.

For a lot of numerical schemes, positivity contradicts conservation, and in this dilemma we need to sacrifice one of them. From (\ref{eq:discrete_lifted}) we learn that if the stability of LP-flow does not rely on positivity, then positivity can be sacrificed. Indeed, the only equation that we are solving numerically is the linear diffusion on spheres, where the diffusion coefficient is always positive semi-definite. 
\end{remark}

\subsection{Forward Euler scheme}
Recall the classical scheme
\begin{equation}\label{eq:discrete_classical}
    f^{(n+1)}_{h} = f^{(n)}_{h} + \sqint_{n\Delta t}^{(n+1)\Delta t } q_{h}f_{h}(\tau) d\tau.
\end{equation}

Schemes \ref{eq:discrete_lifted} and \ref{eq:discrete_lifted} are equivalent only when using forward Euler scheme. In that case we have
\begin{equation*}
    \pi \sqint_{n\Delta t}^{(n+1)\Delta t } Q_{h}F_{h}(\tau) d\tau = \Delta t \pi Q_{h} (f^{(n)}_{h}\otimes f^{(n)}_{h}),
\end{equation*}
and
\begin{equation*}
    \sqint_{n\Delta t}^{(n+1)\Delta t } q_{h}f_{h}(\tau) d\tau = \Delta t q_{h} f_{h}^{(n)}.
\end{equation*}
Equivalence can be easily verified when the spatial discretizations $Q_{h}$ and $q_{h}$ satisfy that
\begin{equation*}
\pi Q_{h} (f^{(n)}_{h}\otimes f^{(n)}_{h}) = q_{h} f_{h}^{(n)}
\end{equation*}

\bigskip
In particular, a Galerkin scheme for $F$ can be written as
\begin{equation*}
\left(F_{h}^{(n+1)}, \varphi_{h}\otimes \psi_{h}\right) = \left(f_{h}^{(n)}\otimes f_{h}^{(n)}, \varphi_{h}\otimes \psi_{h}\right) - \Delta t B_{h}\left(f_{h}^{(n)}\otimes f_{h}^{(n)}, \varphi_{h}\otimes \psi_{h} \right).
\end{equation*}

The above linear system has a huge degree of freedom, but since we are not interested in the full vector $F^{(n+1)}_{h}$, it suffices to solve a sub-system. Indeed, suppose we have a set of orthogonal basis $\varphi_{i}(\mathbf{v})\psi_{j}(\mathbf{w})$, with $\psi_{0}=1_{\Omega_{L}}$, where $\Omega_{L}\subset \mathbb{R}^{d}$ is the cut-off domain, then  
\begin{equation*}
    F_{h}^{(n+1)} = \sum_{i=0}^{N}\sum_{j=0}^{N} F_{i,j}^{(n+1)}\varphi_{i}(\mathbf{v})\psi_{j}(\mathbf{w}),
\end{equation*}
and 
\begin{equation*}
    \pi F_{h}^{(n+1)} \coloneqq \int_{w}F_{h}^{(n+1)} = \sum_{i=0}^{N}\sum_{j=0}^{N} F_{i,j}^{(n+1)}\varphi_{i}(\mathbf{v})\delta_{j,0}\left\lVert \psi_{0}\right\rVert^{2}_{L^{2}} = \left\lVert \psi_{0}\right\rVert^{2}_{L^{2}}\sum_{i=0}^{N}F_{i,0}^{(n+1)}\varphi_{i}(\mathbf{v}) .
\end{equation*}
The $N \times 1$ vector $F_{i,0}^{(n+1)}$ is just a sub-vector of $N^{2}\times 1$ vector $F_{i,j}^{(n+1)}$. 

Let $\psi_{h} = \psi_{0}$, we obtain the equation for $\pi F_{h}^{n+1}$:
\begin{equation*}
    \left(\pi F_{h}^{(n+1)}, \varphi_{h}\right) = \left(f_{h}^{(n)}, \varphi_{h}\right) - \Delta t B_{h}\left(f_{h}^{(n)}\otimes f_{h}^{(n)}, \varphi_{h}\otimes 1\right).
\end{equation*}
In other words, any scheme that can be written in the following form belongs to the lifting-projection scheme family:
\begin{equation*}
    \left(f_{h}^{(n+1)}, \varphi_{h}\right) = \left(f_{h}^{(n)}, \varphi_{h}\right) - \Delta t B_{h}\left(f_{h}^{(n)}\otimes f_{h}^{(n)}, \varphi_{h}\otimes 1\right).
\end{equation*}
For example, the finite difference method for relativistic Landau equation by Shiroto and Sentoku \cite{shiroto2019structure}, the spectral method for Landau equation by Pareschi et. al. \cite{pareschi2000fast}, and the Petrov-Galerkin method for Boltzmann equation by Gamba and Rjasanow \cite{gamba2018galerkin}.


\subsection{Backward Euler scheme}
Recall the classical backward Euler schemes
\begin{equation*}
    f^{(n+1)}_{h} = f^{(n)}_{h} + \Delta t q_{h} f_{h}^{(n+1)}.
\end{equation*}
They have a equivalent lifted version as follows:
\begin{equation*}
    f^{(n+1)}_{h} \otimes f^{(n+1)}_{h}= f^{(n)}_{h}\otimes f^{(n)}_{h} + \Delta t Q_{h} (f_{h}^{(n+1)}\otimes f_{h}^{(n+1)}).
\end{equation*}
Using backward Euler on the lifted equation, we have
    \begin{equation*}
        F^{(n+1)}_{h} = f^{(n)}_{h}\otimes f^{(n)}_{h} + \Delta t Q_{h} F_{h}^{(n+1)}.
    \end{equation*}

In general, the solution $F^{(n+1)}_{h}$ cannot be expressed as a rank-$1$ function like $f^{(n+1)}_{h} \otimes f^{(n+1)}_{h}$, therefore $\pi F_{h}^{(n+1)} \neq f_{h}^{(n+1)}$.

In conclusion, backward Euler schemes for the lifted Landau/Boltzmann equation cannot be the lifted version of any existing scheme for the Landau/Boltzmann equation. The following example reveals the difference.

\bigskip
Consider the Boltzmann equation associated to VHS model. By definition of the backward Euler scheme:
    \begin{equation}\label{eq:imp_Boltz}
        F^{(n+1)} - \widetilde{f}^{(n)}\otimes \widetilde{f}^{(n)} = \Delta t \rho^{\gamma} AF^{(n+1)} - \Delta t \rho^{\gamma} F^{(n+1)}.
    \end{equation}
Apply spherical averaging operator $A$ on both sides, note that since $AA = A$, we have
    \begin{equation*}
        AF^{(n+1)} - A\left(\widetilde{f}^{(n)}\otimes \widetilde{f}^{(n)}\right) = \Delta t\rho^{\gamma} AF^{(n+1)} - \Delta t \rho^{\gamma}AF^{(n+1)} \Rightarrow AF^{(n+1)} = A\left(\widetilde{f}^{(n)}\otimes \widetilde{f}^{(n)}\right).
    \end{equation*}
Substitute the above equation back into (\ref{eq:imp_Boltz}), we obtain
    \begin{equation*}
        F^{(n+1)} - \widetilde{f}^{(n)}\otimes \widetilde{f}^{(n)} = \Delta t\rho^{\gamma} A\left(\widetilde{f}^{(n)}\otimes \widetilde{f}^{(n)}\right) - \Delta t\rho^{\gamma} F^{(n+1)}.
    \end{equation*}
It follows that
    \begin{equation*}
        \begin{split}
            F^{(n+1)} =& \frac{1}{1+\Delta t\rho^{\gamma}}\widetilde{f}^{(n)}\otimes \widetilde{f}^{(n)} + \frac{\Delta t\rho^{\gamma}}{1+\Delta t\rho^{\gamma}}A\left(\widetilde{f}^{(n)}\otimes \widetilde{f}^{(n)}\right)\\
            =&\widetilde{f}^{(n)}\otimes \widetilde{f}^{(n)} + \frac{\Delta t\rho^{\gamma}}{1+\Delta t\rho^{\gamma}}\left(A\left(\widetilde{f}^{(n)}\otimes \widetilde{f}^{(n)}\right)-\widetilde{f}^{(n)}\otimes \widetilde{f}^{(n)}\right)\\
            =&\widetilde{f}^{(n)}\otimes \widetilde{f}^{(n)} + \frac{\Delta t\rho^{\gamma}}{1+\Delta t\rho^{\gamma}}\rho^{-\gamma}Q\left(\widetilde{f}^{(n)}\otimes \widetilde{f}^{(n)}\right),
        \end{split}
    \end{equation*}
or equivalently 
    \begin{equation*}
        \widetilde{f}^{(n+1)} = \widetilde{f}^{(n)} +\pi \frac{\Delta t}{1+\Delta t\rho^{\gamma}}Q\left(\widetilde{f}^{(n)}\otimes \widetilde{f}^{(n)}\right).
    \end{equation*}
In particular, when $\gamma = 0$, we have
\begin{equation*}
    \widetilde{f}^{(n+1)} = \widetilde{f}^{(n)} +\frac{\Delta t}{1+\Delta t}q\left(\widetilde{f}^{(n)}\right)
\end{equation*}
Now we can see that the backward Euler LP scheme actually gain stability through a rescaled time step.

\begin{remark}
    The backward Euler LP scheme can also be written in variational form as follows:
    \begin{equation*}
        F^{(n+1)} = \argmin\limits_{F} \left[\lVert F - \widetilde{f}^{(n)}\otimes \widetilde{f}^{(n)}\rVert^{2}_{L^{2}} + 2\Delta t \left( \rho^{\gamma/2}\nabla_{\sigma} F, \rho^{\gamma/2}\nabla_{\sigma} F\right)\right]
    \end{equation*}
    This form is analogous to the variational approximation schemes based on gradient flow in \cite{erbar2023gradient, carrillo2024landau}. In their papers, the energy functional containing entropy is fully nonlinear and requires positivity, while in our paper the energy functional is bilinear and well-defined for signed functions.
\end{remark}

\subsection{Green's function method}
Recall the analytical representation of the LP solutions (\ref{eq:3dLandau_LP}) and (\ref{eq:VHSBoltzmann_LP}). Since both of them are in the form of a Boltzmann operator, it suffices to take the 3d Landau equation as an example.

For 3d Landau equation we have
\begin{equation*}
    \widetilde{f}^{(n+1)} (\mathbf{v}) = \int_{w}\int_{\mathbb{S}^{2}} K(\rho,\widehat{\mathbf{u}}\cdot \mathbf{\zeta},\Delta t)\widetilde{f}^{(n)}(\mathbf{v}') \widetilde{f}^{(n)}(\mathbf{w}') d\zeta.
\end{equation*}
where
\begin{equation*}
    K(\rho,\zeta\cdot\sigma, t) = \frac{1}{4\pi}\sum_{l=0}^{\infty} \exp(-l(l+1)4\rho^{\gamma}t)P_{l}(\zeta\cdot\sigma).
\end{equation*}

In \cite{gamba2017fast} the authors proposed a fast spectral method for the computation of a general Boltzmann collision operator. The cost is in the order of $\mathcal{O}(m n^{4}\log n)$, where $n$ is the number of modes in each dimension and $m$ is the number of quadrature points on a sphere. That method can be applied here with no extra effort.

\begin{remark}
    For comparison, computing the Landau collision operator with spectral method \cite{pareschi2000fast} takes $\mathcal{O}(n^{3}\log n)$ flops. The evaluation part is faster than the discrete semigroup method proposed here. However, its time step is restricted by stability considerations. As we have mentioned in the subsection on forward Euler scheme, the time step size is restricted by the CFL condition for diffusion equations, i.e. $\Delta t \lesssim \mathcal{O}(\frac{1}{n^{2}})$, while the Green's function method is unconditionally stable.
\end{remark}

\section{Numerical verification}\label{sec:verify}
In \ref{subsec:properties}, through Proposition \ref{prop:conservation} and Theorem \ref{thm:entropy} we have proved that the consecutive LP flow admits conservation laws and decay of entropy. In what follows we will verify those properties numerically. Consider the Boltzmann equation with Knudsen number $\varepsilon$:
\begin{equation*}
    \partial_{t} f = \frac{1}{\varepsilon}q(f)  = \frac{1}{\varepsilon} \Pi \lvert \mathbf{u}\rvert^{\gamma} (\mathcal{A}-I)(f\otimes f),
\end{equation*}
the associated lifted Boltzmann equation reads
\begin{equation}\label{eq:knudsen}
    \partial_{t} F = \frac{1}{\varepsilon}QF =\frac{1}{\varepsilon}\lvert \mathbf{u} \rvert^{\gamma} \left(\mathcal{A} F - F\right).
\end{equation}

Given initial condition $F(\mathbf{v}, \mathbf{w}, 0) = F_{0}(\mathbf{v}, \mathbf{w})$, the explicit solution to Equation~\eqref{eq:knudsen} is
\begin{equation*}
    F(\mathbf{v}, \mathbf{w}, t) = \exp(\frac{1}{\varepsilon}Qt) F_{0} = e^{-\lvert \mathbf{u} \rvert^{\gamma}t/\varepsilon} F_{0} + \left(1-e^{-\lvert \mathbf{u} \rvert^{\gamma}t/\varepsilon}\right) \mathcal{A} F_{0}.
\end{equation*}
Define the relaxed forward Euler mapping based on the above formula:
\begin{equation}\label{eq:relaxed_Euler}
    \Phi_{\tau}\left(f_{h}^{(s)}\right) \coloneqq f^{(s)}_{h} + \Pi (1-e^{-\tau \lvert \mathbf{u} \rvert^{\gamma}/\varepsilon})(\mathcal{A}-I)(f_{h}^{(s)}\otimes f_{h}^{(s)})
\end{equation}
It reduces to standard forward Euler method when $\tau \rightarrow 0$, moreover, it has been proved that for any Knudsen number $\varepsilon$ and any time step $\tau$, 
\begin{itemize}
    \item $\Phi_{\tau}\left(f_{h}^{(s)}\right)$ has the same mass, momentum and energy as $f_{h}^{(s)}$.
    \item $\Phi_{\tau}\left(f_{h}^{(s)}\right)$ preserves positivity, and has lower entropy than $f_{h}^{(s)}$.
\end{itemize}

To verify the above properties numerically, let the initial condition be
\begin{equation} \label{eq:twoG3d}
    f_{0}(\mathbf{v})=\exp(-\frac{|\mathbf{v}-\mathbf{v}_{1}|^2}{2}) + \exp(-\frac{|\mathbf{v}-\mathbf{v}_{2}|^2}{2})
\end{equation}
where $\mathbf{v}_{1} = (2, 0, 0)$ and $\mathbf{v}_{2} = (0, -2, 0)$. We solve the Boltzmann equation with $\gamma = 1$ and $\varepsilon = 0.001$ using both the standard and the relaxed forward Euler method. The results are shown in Figure \ref{fig:AP}. It can be observed that standard forward Euler becomes unstable, while the relaxed forward Euler (the consecutive LP flow) is stable.

\begin{figure}[H]
    \centering
    \begin{subfigure}[b]{0.3\textwidth}
        \centering
        \includegraphics[width=\textwidth]{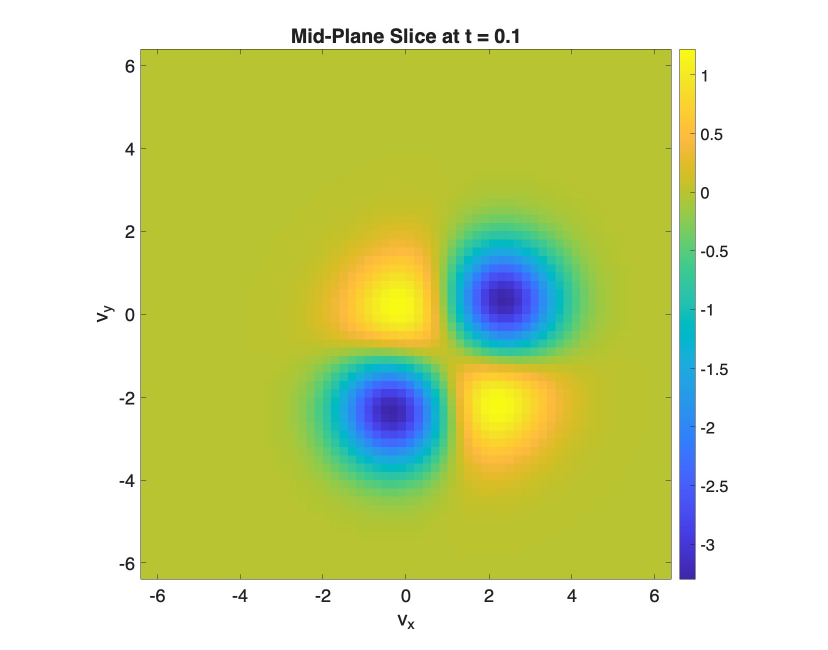}
        \caption{standard method, at $t = 0.1$.}
    \end{subfigure}
    \hfill
    \begin{subfigure}[b]{0.3\textwidth}
        \centering
        \includegraphics[width=\textwidth]{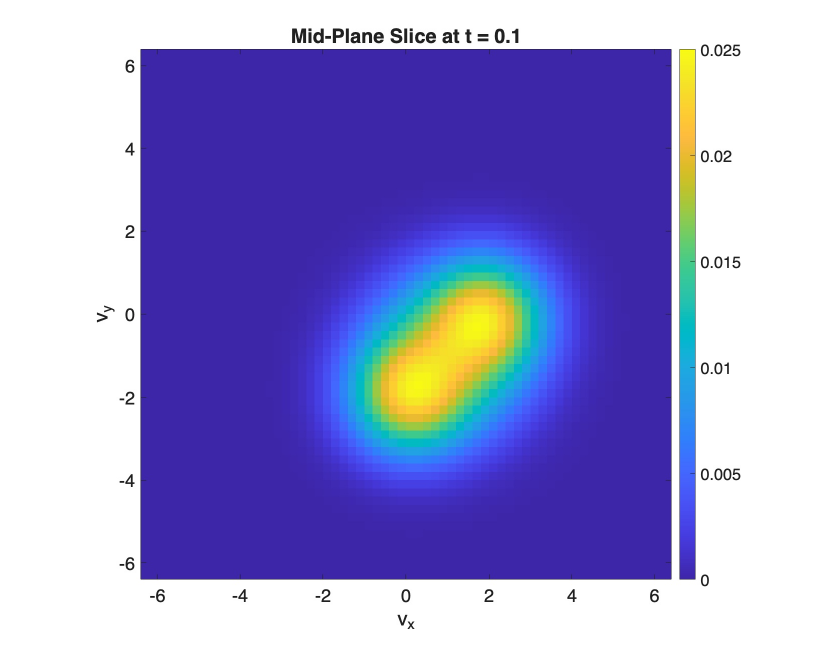}
        \caption{relaxed method, at $t=0.1$.}
    \end{subfigure}
    \hfill
    \begin{subfigure}[b]{0.3\textwidth}
        \centering
        \includegraphics[width=\textwidth]{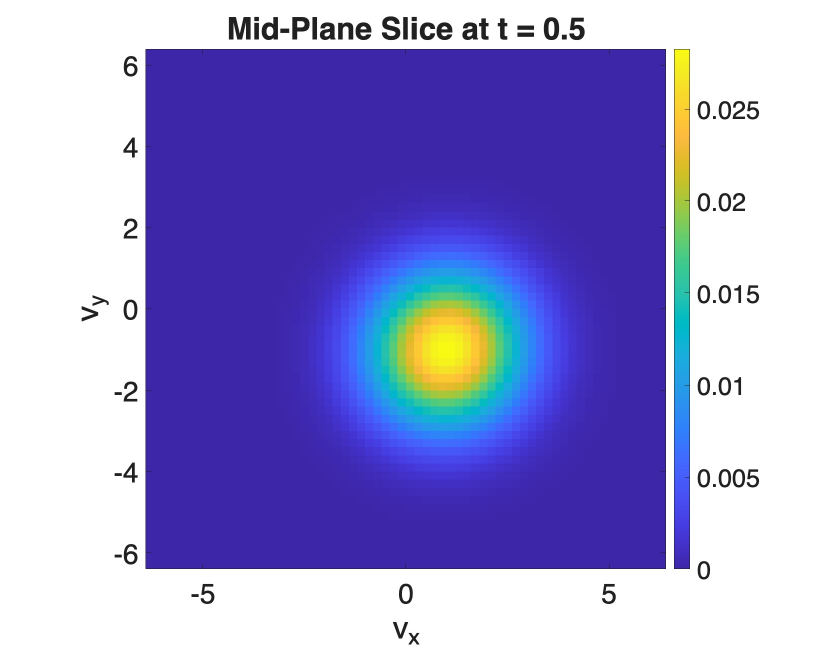}
        \caption{relaxed method, at $t=0.5$.}
    \end{subfigure}
    \begin{subfigure}[b]{0.45\textwidth}
        \centering
        \includegraphics[width=\textwidth]{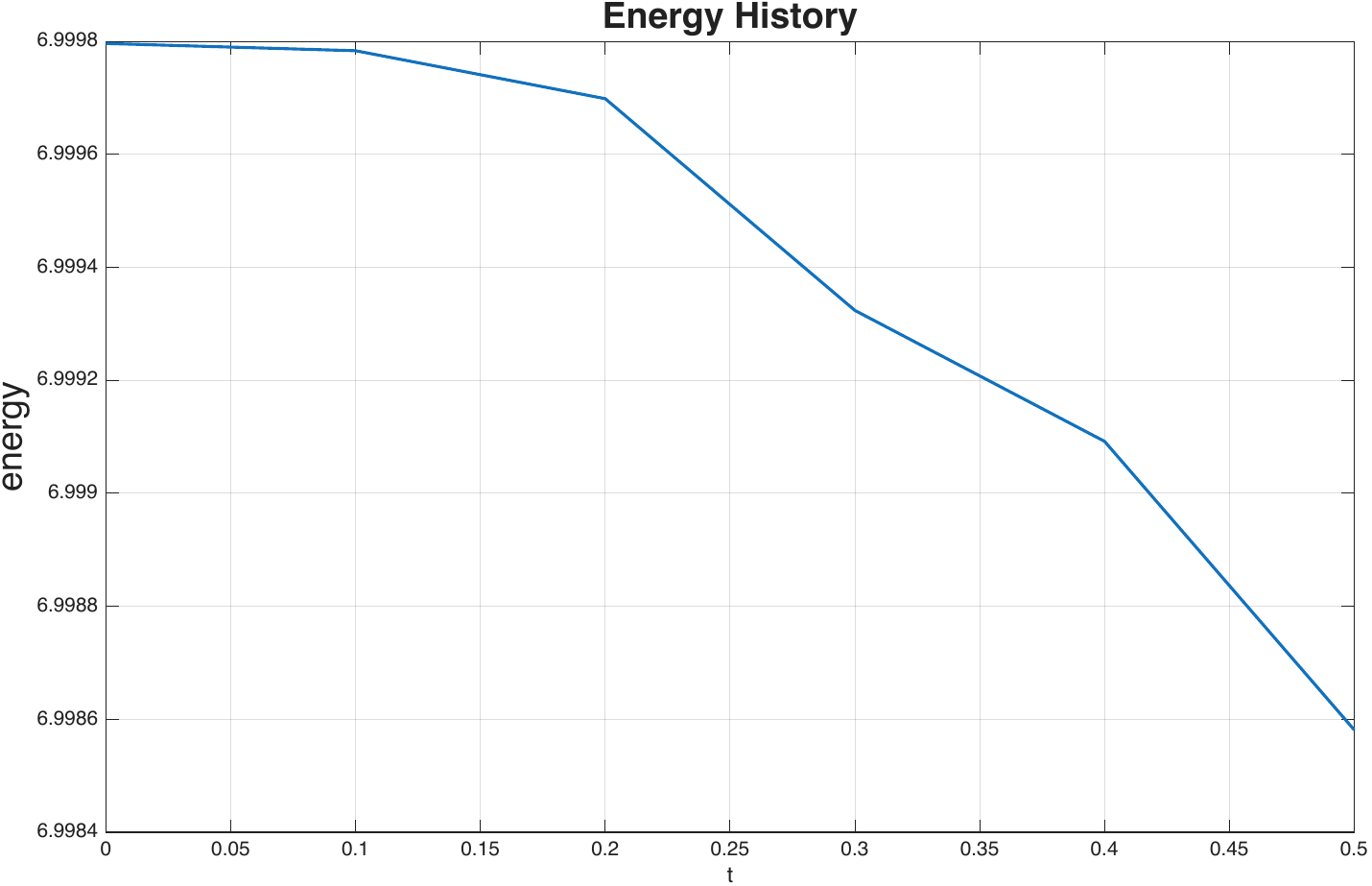}
        \caption{Energy evolution, relaxed method}
    \end{subfigure}
    \hfill
    \begin{subfigure}[b]{0.45\textwidth}
        \centering
        \includegraphics[width=\textwidth]{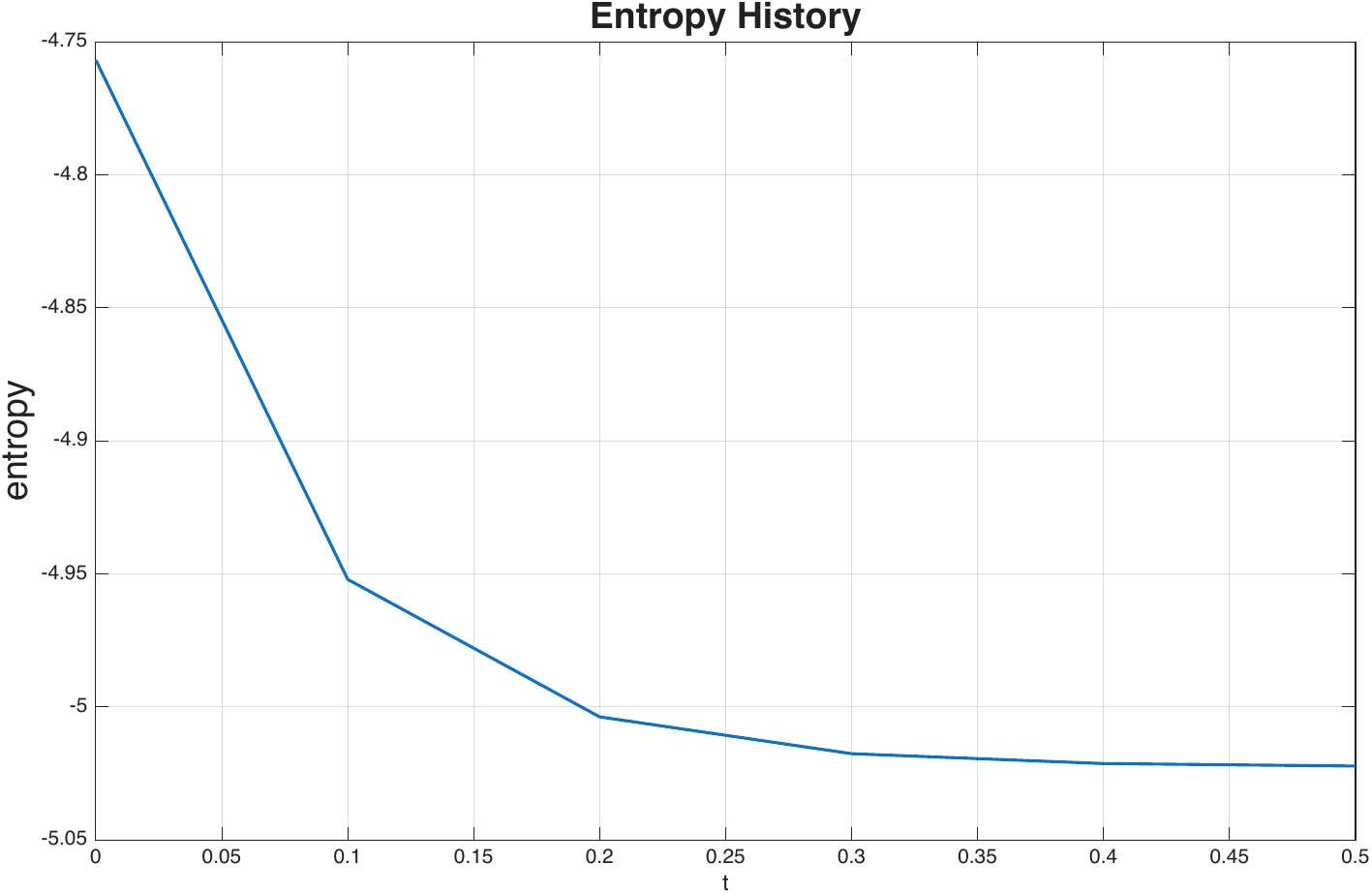}
        \caption{Entropy evolution, relaxed method}
    \end{subfigure}
    \caption{Numerical results with $(h,\Delta t) = (0.2, 0.1)$. The solutions are sliced at $v_{z} = 0$.}
    \label{fig:AP}
\end{figure}

The relaxed forward Euler method can be a building block of novel asymptotic-preserving time discretization schemes for the spatially inhomogeneous Boltzmann equation, which is an interesting future direction to explore.

\section{Conclusion}\label{sec:conclusion}
We have proposed the consecutive lifting–projection flow as a new approximation framework for the Boltzmann and Landau equations. By lifting the nonlinear collision dynamics to a linear master equation in a higher-dimensional phase space and projecting back to the original velocity variables, the LP flow reveals a hidden linear structure underlying these classical kinetic models.

The LP framework possesses a lot of favorable properties. It preserves the fundamental conservation laws of mass, momentum, and energy, satisfies an entropy dissipation principle, and converges to the correct Maxwellian equilibrium. For Maxwell molecules, we established a quantitative $L^1$ error bound for the consecutive LP flow, demonstrating that the global approximation error is first order. The lifted formulation also admits explicit semigroup representations, which enable both analytical insight and practical numerical constructions.

From a numerical perspective, the lifting–projection viewpoint unifies a broad class of existing deterministic schemes. Any method that can be interpreted as solving a lifted linear Kac-type equation followed by projection naturally fits into this framework. This reinterpretation provides a new lens for stability analysis. Moreover, the LP framework clarifies the role of positivity, showing that when positivity conflicts with exact conservation, stability may still be retained without enforcing pointwise non-negativity.

The framework also opens new directions for method development. We emphasize that the backward Euler and Green’s function methods discussed in this work form a non-exhaustive list of possible numerical schemes, and further approaches can be constructed.

\bibliographystyle{plain}
\bibliography{main.bib}
\end{document}